\documentclass[a4paper,10pt]{article}
\usepackage[utf8]{inputenc}
\usepackage[T1]{fontenc}

\usepackage{amsmath,amsthm,amssymb,amsfonts}
\usepackage{mathtools}
\usepackage[a4paper]{geometry}
\usepackage{a4wide}
\usepackage[capitalise]{cleveref}

\usepackage{thmtools} 
\usepackage{thm-restate}

\usepackage{enumitem}

\usepackage{graphicx}
\usepackage{mathabx}

\usepackage{nicefrac}

\usepackage{calrsfs}
\DeclareMathAlphabet{\pazocal}{OMS}{zplm}{m}{n} 

\usepackage{dsfont} 

\newtheorem{theorem}{Theorem}[section]
\newtheorem{lemma}[theorem]{Lemma}
\newtheorem{proposition}[theorem]{Proposition}
\newtheorem{corollary}[theorem]{Corollary}

\newtheorem{example}[theorem]{Example}
\newtheorem{remark}[theorem]{Remark}

\newtheorem*{conjecture}{Conjecture}

\newcommand{\R}{\mathbb{R}}
\newcommand{\bN}{\mathbb{N}}
\newcommand{\cC}{\mathcal{C}}
\newcommand{\cH}{\mathcal{H}}
\newcommand{\cA}{\mathcal{A}}
\newcommand{\eps}{\varepsilon}
\newcommand{\cO}{\mathcal{O}}
\newcommand{\spt}{\text{spt}}

\newcommand{\id}{\mbox{\textup{id}}}

\newcommand{\mres}{\mathbin{\vrule height 1.6ex depth 0pt width
0.13ex\vrule height 0.13ex depth 0pt width 1.3ex}} 

\newcommand*\into[1]{\mbox{\text{int}}(#1)}

\providecommand{\keywords}[1]{\textbf{Keywords:} #1}
\providecommand{\MSCtwoK}[1]{\textbf{2010 Mathematics Subject Classification:} #1}



\title{Besicovitch-Federer projection theorem for mappings having constant rank of the Jacobian matrix}
\author{Jacek Adam Ga\l{}\k{e}ski}
\date{\today}
\begin{document}
\maketitle
\begin{abstract}
The purpose of this article is to prove a generalisation of the Besicovitch-Federer projection theorem about a
characterisation of rectifiable and unrectifiable sets in terms of their projections.
For an $m$-unrectifiable set $\Sigma\subset\R^n$ having finite Hausdorff measure and $\eps>0$,
we prove that for a mapping $f\in\cC^1(U,\R^n)$ having constant, equal to $m$, rank of the Jacobian matrix there exists a mapping $f_\eps$ whose rank of the Jacobian matrix is also constant, equal to $m$, such that $\|f_\eps-f\|_{\cC^1}<\eps$ and $\cH^m(f_\eps(\Sigma))=0$. We derive it as a consequence of the Besicovitch-Federer theorem stating that the $\cH^m$ measure of a generic projection of an $m$-unrectifiable set $\Sigma$ onto an $m$-dimensional plane is equal to zero.
\end{abstract}
\keywords{Purely unrectifiable set, Hausdorff measure, rank of the Jacobian matrix}\newline
\MSCtwoK{(Primary) 28A75, (Secondary) 57N20}
\section{Introduction}

Rectifiability is one of the most important concepts investigated in geometric measure theory. \mbox{A rectifiable} set coupled with an orientation and integer multiplicity 
results in a notion of an \emph{integral current} which has two nice properties. First, the space of such currents enjoys a compactness property proved 
by Federer and Fleming in \cite{ff60}, and second, it contains the class of smooth orientable surfaces, and thus constitutes its natural generalisation.

The subject of this work originates from a modification of a systematic way of projecting cubes of Whitney partition onto their $m$-dimensional skeletons, 
see \cite[David and Semmes, Chapter 3]{unif-rect}. Let $\cH^m$ denote the $m$-dimensional Hausdorff measure. 
In \cite[2.9]{almExi} Almgren constructs a similar smooth mapping $l_m$ from almost all of $\R^n$ to 
the $m$-dimensional skeleton of cubes. Then, in Section 2.9(b) on the page 338 of the same work, Almgren claims that 
for a purely $m$-unrectifiable set $\Sigma$ with $\cH^m(\Sigma)<\infty$, there exists a perturbation $l^\ast_m$ arbitrarily close to $l_m$  in $\cC^1$ topology such that $\cH^m(l^\ast_m(\Sigma))=0$. We provide a proof of this claim for ``\emph{an arbitrary mapping having constant rank of the Jacobian matrix}'' instead of the map $l_m$.

The proof is based on a construction of a diffeomorphism $\Xi_\eps$ such that $\cH(f\circ\Xi_\eps(\Sigma))=0$. 
The significance of class $\cC^1$ comes from the potential application. A map of class $\cC^1$ allows for a continuous \emph{pushforward} of integral currents. This may be a useful tool in proving the existence of solutions, in this class, to variational problems, for example the Plateau problem. In this particular problem removing the unrectifiable part by a slight perturbation strictly decreases the measure, so it seems like a step in the right direction.

A similar problem has been recently studied by Pugh in \cite{Pughlocal}.
In this work Pugh investigates Lipschitz mappings arbitrarily close to identity in $\cC^0$ topology such that the measure of the image under this 
mappings of an unrectifiable part is arbitrarily small. Moreover, if this unrectifiable part was a part of a larger  set, the rectifiable part of the 
larger set is changed only 
slightly under this mapping. Note that Pugh constructs a Lipschitz map that alone abates the measure of unrectifiable part. We, on the other hand, build 
a diffeomorphism of the domain of a given map $f$ such that after a composition with $f$ the measure of the unrectifiable part is \emph{zero}.

Let $n,m\in\bN$ be such that $m<n$.  We say the set $E\subset\R^n$ is $m$-\emph{rectifiable} if there exist a countable family of Lipschitz maps 
$f_i:\R^m\rightarrow\R^n$ for $i\in\bN$, such that 
$\cH^m(E\setminus\bigcup
f_i(\R^m))=0.$ 
We say that the set $\Sigma$ is \emph{purely} $m$-\emph{unrectifiable} if 
$\cH^m(\Sigma\cap E)=0$ for every $m$-rectifiable set $E$.

The Besicovitch-Federer projection theorem was first proved by Besicovitch in \cite{Besico3} for one dimensional sets on the plane and then generalised to 
any dimension by Federer in \cite{Fed47}. We use the formulation given in \label{clasBF}\cite[Theorem 18.1]{Mattila1} which consists of two statements. 
The first one characterises rectifiable sets: a set $A$ is $m$-rectifiable if and only if the image of every subset $B\subset A$ of positive $\cH^m$-measure under
a generic projection has positive $\cH^m$-measure, i.e. $\cH^m(P_V(B))>0$. The second one concerns unrectifiable sets and states that for a generic orthogonal 
projection $P_V$ onto an $m$-dimensional plane $V$ the measure of the image of a purely $m$-unrectifiable set $\Sigma$ is zero, i.e. $\cH^m(P_V(\Sigma))=0$.
The statement about rectifiable sets expresses the intuitions: on the plane, the shadow cast by an interval has 
positive length unless the interval is parallel to the light rays. Moreover, we see that that $1$-unrectifiable sets despite having positive $\cH^1$-measure cast shadow of length zero for almost all directions of light rays.

On the pages that follow we generalise the second of the above statements. We consider the consequences of the substitution of the orthogonal projections 
in the B-F theorem by the class of mappings having constant rank of the Jacobian matrix. For an open set $U\subset \R^n$, we denote this class by
\begin{equation}
 \cC^1_{=m}(U,\R^n)\coloneqq\left\lbrace f\in \cC^1(U,\R^n)\:\vert\: \mathrm{dim}\:\mathrm{im}\: Df(x) = m,\forall x\in U \right\rbrace.
\end{equation}
This is a much wider class of mappings, and consequently our assertion differs from the result of the theorem of Besicovitch and Federer. Instead of a measure 
theoretic result we obtain \mbox{a topological} one involving the density. Precisely, in \cref{sec3} we prove the following theorem.
\begin{restatable*}{thm}{main}
\label{thm:main}
 For a purely $m$-unrectifiable set $\Sigma$ contained in an open set $U\subset \R^n$, \newline such that 
 $\mathcal{H}^m(\Sigma)<\nobreak\infty$, 
 let 
 $$\mathcal{A}(\Sigma)\coloneqq\left\lbrace f\in\cC^1(U,\R^n)\:\vert\: \mathrm{dim}\:\mathrm{im}\: Df = m\text{, } \mathcal{H}^m(f(\Sigma))>0  \right\rbrace.$$
 Then the interior of the set $\mathcal{A}(\Sigma)$
 in the $\cC^1$ topology is empty.
\end{restatable*}
The proof is based on a construction of a diffeomorphism of the domain that can be arbitrarily close to the identity, and such that the pullback of 
$f\in\mathcal{A}(\Sigma)$ by this diffeomorphism is no longer in the set $\mathcal{A}(\Sigma)$.
This construction is divided into two main steps. First, in \cref{sec2}, we prove a 
local variant of this theorem. 
The proof of local variant is based on the fact that for a mapping $f\in\cC^1_{=m}(U,\R^n)$ and a point $x\in U$ there exists a neighbourhood $V$ of $x$ 
and diffeomorphisms $\phi$ and $\psi$ of $\R^n$ such that $f{\vert}_V=\psi^{-1}\circ\pi\circ\phi{\vert}_V$, where $\pi$ is an orthogonal projection onto 
some $m$-plane.
We insert a rotation $\theta$ of 
$\R^n$ into this composition to obtain
$$f_\eps=f\circ\Xi=(\psi^{-1}\circ\pi\circ\phi)\:\circ\:(\phi^{-1}\circ\theta\circ\phi).$$ 
We then show that $\theta$ can be chosen so that $f_\eps$ is 
closer to $f$ than $\eps$ in $\cC^1$ topology and zeroes the measure of unrectifiable $\Sigma$ in the image 
, i.e. $\cH^m(f_\eps(\Sigma\cap V))=0$. 

Next, in \cref{sec3} we prove the main theorem. 
The crucial difficulty is to glue diffeomorphisms $\Xi=\phi^{-1}\circ\theta\circ\phi$, that do not agree on intersections and act non-trivially on 
boundaries of their individual domains, constructed in the first step of the proof. 
This is achieved by a careful construction of a countable family $\mathcal{F}$ of open, connected, 
pairwise disjoint sets, whose closures covers the domain $U$. 
Then, we iteratively apply the modified diffeomorphisms constructed in the local variant of the theorem to elements of the family $\mathcal{F}$.

In order to obtain this modification we interpolate between diffeomorphism from the local variant and the identity using a vector flow.
The result is that away from the boundary of an element of the family $\mathcal{F}$ the modified diffeomorphism is equal to the one from the local variant, 
and on the boundary it is equal to the identity. Once we have applied this modified diffeomorphism, the part of the set $\Sigma$ that was away from the boundary, after further composition with $f$, yields no measure.  

It is not the case when we consider the part of the set $\Sigma$ that is near the boundary. On the collar around the boundary the modified diffeomorphism is not equal to the mapping from the local variant and consequently we do not know the outcoming measure of the
perturbed $\Sigma$ after the composition with $f$.
The solution is to iterate modified diffeomorphisms as described in previous paragraph on decreasingly thinner collars around the boundaries.

%
%

Instead of $\cC^1_{=m}(U,\R^n)$ we may consider even a wider class of mappings having bounded rank of the Jacobian matrix, namely 
\begin{equation}
 \cC^1_{\leq m}(U,\R^n)\coloneqq\left\lbrace f\in \cC^1(U,\R^n)\:\vert\: \mathrm{dim}\:\mathrm{im}\: Df(x) \leq m,\forall x\in U \right\rbrace.
\end{equation}
In \cref{sec:lowrank} we discuss the potential extension of the main theorem. The argument would be \mbox{a simple} consequence of Sard's theorem under the following conjecture. 
 \begin{conjecture}
 \begin{equation}
 \cC^\infty_{\leq m}(U,\R^n)\text{ is dense in }
 \cC^1_{\leq m}(U,\R^n).
\end{equation}
\end{conjecture}
\noindent
The difficulty in proving this conjecture is that for $f\in\cC^\infty_{\leq m}(U,\R^n)$ the convolution with the mollifier $\vartheta_\eps\star f\notin\cC^\infty_{\leq m}(U,\R^n)$ because rank of the Jacobian matrix is generally no longer bounded by $m$.  
\subsection{Notation and conventions}

\begin{enumerate}[label={\arabic*.}]
 \item For $r\in\R$ and $A\subset \R^n$ the $r$-neighbourhood of the set $A$ is $$B_r(A)\coloneqq\bigcup_{p\in A}B(p,r).$$
 \item The \emph{Standard mollifier} $\vartheta$ is given by  
\begin{equation*}
 \vartheta(x)\coloneqq \left\lbrace
\begin{array}{ll}
 \nicefrac{\exp\left(\nicefrac{-1}{(1-|x|^2)}\right)}{I_n} &\mbox{for }|x|<1 \\
 0 &\mbox{for } |x| \geq 1
\end{array}
\right.,
\end{equation*}
where $I_n$ is such that $\smallint_{\R^n}\vartheta(x) dx=1$ and also denote $$\vartheta_\eps(x)\coloneqq\eps^{-n}\vartheta\left(\eps^{-1}x\right).$$
\end{enumerate}


\section{Local variant}\label{sec2}
Before we discuss the local variant of the theorem we shall explain how the orthogonal projections and the class $\cC^1_{=m}(U,\R^n)$ are related to each other.

For a continuously differentiable map $f:\R^n\rightarrow\R^m$ with 
$r(Df)\coloneqq\text{dim}\:\text{im}(Df)=m$ at some point $y$ there exist \cite[Th. 2-13 p.43]{Spivak} a neighbourhood $V\ni y$ and a mapping 
$h:V\rightarrow\R^n$ such that for $(x_1,\ldots,x_n)\in V$ we have
\begin{equation}\label{wlokn}
 f\circ h (x_1,x_2,\ldots,x_n)=(x_{n-m+1},x_{n-m+2},\ldots,x_n).
\end{equation}
In other words we say: there exist a diffeomorphism that straightens preimages $f^{-1}(y)$ for a $y\in \R^m$. Its inverse is the 
diffeomorphism $h$ mentioned above.

The image of a mapping $f$ from $\cC^k_{=m}(\R^n,\R^n)$
is locally an $m$-dimensional manifold, thus there exists a diffeomorphism 
$\psi:W\rightarrow \R^n$ of a neighbourhood $W\subset\R^n$
of the point $f(a)$ such that $\psi(f(U)\cap W)$ is contained in the plane 
$\lbrace (x_1,\ldots,x_n)\in\R^n\:\vert\: x_{m+1}=x_{m+2}=\ldots= x_n=0\rbrace$. 
By combining this observation and \eqref{wlokn} we obtain the classical result.
\begin{corollary}[Constant rank theorem]\label{col1}
Suppose $f\in\cC^k_{=m}(U,\R^n)$.  For any $x\in U$ there exist open sets $U_x \ni x$ and $W_{f(x)} \ni f(x)$, and diffeomorphisms $\phi_x:U_x\rightarrow \R^n$ and $\psi_{f(x)}:W_{f(x)}\rightarrow \R^n$ of class $\cC^k$ such that the mapping $(\psi_{f(x)}\circ f\circ {\phi_x}^{-1}) \colon\phi_x(U_x)\rightarrow \R^n$ satisfies:
\[
 (\psi_{f(x)}\circ f\circ {\phi_x}^{-1}) = P_V\vert_{\phi_x(U_x)},
\]
where $P_V:\R^n\rightarrow\R^n$ is the projection on an $m$-dimensional plane $V$.
\end{corollary}
\noindent Therefore any map $f\in\cC^k_{=m}(U,\R^n)$ is similar in the above \emph{local} manner, to the orthogonal projection onto an 
$m$-dimensional plane. Therefore, the class $\cC^1_{=m}(U,\R^n)$ is a natural generalisation of projections.

Our temporary goal is the following \emph{local} form of the main theorem.
\begin{theorem}[Local variant]\label{localver}
Let $\eps>0$, $\Sigma$ be an $m$-unrectifiable set, such that $\cH^m(\Sigma)<\infty$ and $f\in \cC^1_{=m}(U,\R^n)$. For any point $x\in U$ there 
exist a neighbourhood $U(x,r)$,
and $f_\eps\in\cC^1_{=m}(U(x,r),\R^n)$ such that $${\|f-f_\eps\|\ }_{\cC^1(U(x,r),\R^n)}\leq \eps\quad\text{and}\quad\cH^m(f_\eps(\Sigma\cap U(x,r)))=0.$$
\end{theorem}
\noindent We precede the proof with a proposition stating that the unrectifiability property is preserved by a bi--Lipschitz diffeomorphism.
\begin{proposition}\label{purunrediff}
 Let $\Sigma\subset U$ be a purely $m$-unrectifiable set with $\cH^m(\Sigma)<\infty$ and $\phi:U\rightarrow V$ be a bi--Lipschitz diffeomorphism. Then 
 $\phi(\Sigma)$ is also an $m$-unrectifiable set.
\end{proposition}
\begin{proof}
The $m$-dimensional Hausdorff measure satisfies the inequality (see \cite[Theorem 7.5, p.103]{Mattila1})
\begin{equation}\label{mesoflip}
 \cH^m(l(A))\leq Lip(l)^m\cH^m(A),
\end{equation}
where $l\colon\R^n\rightarrow\R^n$ is a Lipschitz map and $A$ is a subset of $\R^n$.
Therefore, for  any $m$-rectifiable set $E\subset V$ we see that 
$$\cH^m(\phi(\Sigma)\cap E)\leq Lip(\phi^{-1})^m\cH^m(\Sigma\cap \phi^{-1}(E))=0,$$
where the last equality follows from the rectifiability of $\phi^{-1}(E)$.
\end{proof}
We are now in a good position to prove the lemma.
\begin{proof}[Proof of \cref{localver}.]
We divide the proof into three parts. First part contains the definition of the diffeomorphism
 $\Xi_\theta$ which composed with original map $f$ will be the mapping $f_\eps=f\circ\Xi_\theta$ form the conclusion of the theorem. In second we investigate the distance between $f$ and modified map. In the last part we check when the image $f\circ\Xi_\theta(\Sigma)$ has $m$-dimensional measure equal to zero. 
\begin{enumerate}[wide, labelwidth=!, labelindent=0pt]
 \item[1.]
Let $U_x$ be an open neighbourhood as in \cref{col1}.  Using a translation we can assume that $0\in \text{int}\left(\phi_x(U_x)\right)$. 
For every point $x\in U$ let $r_x=\text{dist}(0,\partial \phi_x(U_x))$. Then for any positive $r<r_x$ we have $\overline{B}(0,r)\subset\subset\phi_x(U_x)$. 
For any such $r$ we define
\begin{equation}
U(x,r)\coloneqq\phi_x^{-1}(B(0,r))\subset U_x.
\end{equation}
For simplicity denote $\phi\coloneqq\phi_x$. For an isometry $SO(n)\ni\theta:\R^n\rightarrow\R^n$ let $\Xi_\theta$ stand for the map

\begin{equation}\label{mapXi}
 \Xi_\theta= \left(\phi^{-1}\circ \theta\circ\phi\right):U(x,r)\rightarrow U(x,r).
\end{equation}
Observe that the further composition with $(\psi\circ f)$ yields a mapping from $U(x,r)$ to $\R^m$. We investigate the image of the set $\Sigma$ under this composition
\begin{equation}
 \begin{split}\label{observ}
  \psi\circ f\circ\Xi_\theta(\Sigma)&=(\psi\circ f\circ\phi^{-1})\circ \theta\circ\phi(\Sigma) \\
  &=(P_V\circ \theta)(\phi(\Sigma)),
 \end{split}
\end{equation}
for some $m$-dimensional plane $V$. We claim that $P_V\circ \theta=\theta\circ P_{\theta^{-1}(V)}$. Indeed, by computing the preimage of a point $v\in V$, we have 
\begin{equation}
 \begin{split}
  (P_V\circ \theta)^{-1}(v)&=\theta^{-1}( {P_V}^{-1}(v))\\
  &=\theta^{-1}(v+V^{\perp})\\
  &=\theta^{-1}(v)+\theta^{-1}(V^{\perp}),
 \end{split}
\end{equation}
while 
\begin{equation}
 \begin{split}
  (\theta\circ P_{\theta^{-1}(V)})^{-1}(v)&=(P_{\theta^{-1}(V)})^{-1}(\theta^{-1}(v))\\
 &=\theta^{-1}(v)+\theta^{-1}(V)^\perp \\
 &=\theta^{-1}(v)+\theta^{-1}(V^\perp).
 \end{split} 
\end{equation}

Therefore, we obtained a projection of the unrectifiable set $\phi(\Sigma)$ onto a subspace $\theta^{-1}(V)$. We will choose symmetry $\theta$
is such a way that it is both close to identity mapping and the projection onto the $m$-plane $\theta^{-1}(V)$ zeroes the measure of $\phi(\Sigma)$.
\item[2.] 
We will prove that there exist an open neighbourhood of $\id$ in $SO(n)$ such that the composition 
\begin{equation}
 f\circ\Xi_\theta:U(x,r)\rightarrow \R^n
\end{equation}
can be as close to the mapping $f\vert_{U(x,r)}$ as we wish.
Observe that both $f$ and $Df$ are continuous up to the boundary of $U(x,r)$ for $r<r_x$, hence, they are uniformly continuous. For the, differential $Df$ define the continuity parameters $\tilde{\eps}$ and $\tilde{\delta}$ by the following statement: for any positive $\tilde{\eps}$ there exists $\tilde{\delta}=\tilde{\delta}(\tilde{\eps})$, such that if $|y-x|\leq \tilde{\delta}$, then $\|Df(y)-Df(x)\|\leq \tilde{\eps}$. We estimate
\begin{equation}\label{stupid}
\begin{split}
 \|f\circ\Xi_\theta - f\|_{\cC^1}&=\|f\circ\Xi_\theta- f\|_{\cC^0}+\|Df(\Xi_\theta).D\Xi-Df\|_{\cC^0} \\
 &\leq\|f\|_{\cC^1}\|\Xi_\theta-\id\|_{\cC^0}+\|Df(\Xi_\theta).(D\Xi-\id+\id)-Df\|_{\cC^0} \\
 &\leq\|f\|_{\cC^1}\|\Xi_\theta-\id\|_{\cC^0}+\|Df(\Xi_\theta).(D\Xi-\id)\|_{\cC^0}+\|Df(\Xi_\theta)-Df\|_{\cC^0} \\
 &\leq\|f\|_{\cC^1}\|\Xi_\theta-\id\|_{\cC^0}+\|Df\|_{\cC^0}\|(D\Xi-\id)\|_{\cC^0} +\|Df(\Xi_\theta)-Df\|_{\cC^0} \\
 &\leq C(\|f\|_{\cC^1})\|\Xi_\theta-\id\|_{\cC^1}+\tilde{\eps}
\end{split}
\end{equation}
for $\|\Xi_\theta-\id\|_{\cC^0}<\tilde{\delta}(\tilde{\eps})$. Observe, that if we pick $\tilde{\eps}<\eps/2$ and $\|\Xi_\theta-\id\|_{\cC^0}<\min(\tilde{\delta},\eps/2)$ then the left hand side of \eqref{stupid} is less than $\eps$. 
Therefore, we need to control the $\cC^1$ norm of $(\Xi_\theta - \id)$. Similarly to \eqref{stupid} one can show that also
$\|\Xi_\theta-\id\|_{\cC^1(U(x,r))}$ is continuous in the argument $\theta$. This provides that for any positive $\eps$ we can find an open ball in $SO(n)$
$$ B_{SO}(\id,\rho)\coloneqq\lbrace\theta\:\vert\:\|\theta-\id\|_{C^1}<\rho\rbrace,$$ 
such that
$$
\|f\circ\Xi_\theta - \id\|_{\cC^1}<\eps\quad \text{for}\quad \theta\in B_{SO}(\id,\rho).
$$
\begin{remark}
 Let $V$ be an $m$-plane in $\R^n $ which is an element of the Grassmann manifold $Gr(n,m)$. Then 
 $$B_{SO}(\id,\rho).V=\lbrace\theta(V)\in Gr(n,m)\:|\:\theta\in B_{SO}(\id,\rho)\rbrace$$
 is the set of the positive $\gamma_{n,m}-$measure. Indeed, let $\nu$ be the uniformly distributed measure on the orthogonal group $O(n)$. 
 From the definition of $\gamma_{n,m}$ we have 
 \begin{equation}
  \begin{split}
   \gamma_{n,m}(B_{SO}(\id,\rho).V)&\overset{def}{=}\nu(\lbrace g\in O(n)\:\vert\: g.V\in B_{SO}(\id,\rho).V \rbrace) \\
   &=\nu(B_{SO}(\id,\rho))>0,
  \end{split}
 \end{equation}
\end{remark}
because $B_{SO}(\id,\rho)$ is an open set also in $O(n)$.

\item[3.]
From Besicovitch-Federer theorem \cite[Theorem 18.1]{Mattila1} in the open set $B_{SO}(\id,\rho).V$ we can find the set $Z_{\phi(\Sigma)}$ consisting of 
those $m$-planes for which 
$$\cH^m(P_{\theta^{-1}(V)}(\phi(\Sigma)))=0\quad \text{for}\quad\theta^{-1}(V)\in Z_{\phi(\Sigma)}$$ 
and such that $\gamma_{n,m}(B_{SO}(\id,\rho).V)=\gamma_{n,m}(Z_{\phi(\Sigma)})$. 
Under the mappings $\psi$ and $\theta$ a set of zero $\cH^m-$measure 
is mapped to a set of zero measure, hence we have 
\begin{equation}
\begin{split}
 0=\cH^m(P_{\theta^{-1}(V)}(\phi(\Sigma)))&=\cH^m(\theta\circ P_{\theta^{-1}(V)}(\phi(\Sigma))) \\
 \text{by }\eqref{observ}\quad&=\cH^m(\psi\circ f\circ\phi^{-1}\circ \theta\circ\phi(\Sigma)) \\
 &=\cH^m(f\circ\Xi_\theta(\Sigma)).
\end{split}
\end{equation}
\end{enumerate}
\end{proof}

\section{The $\cC^1_{=m}$ Besicovitch-Federer Theorem}\label{sec3}

\subsection{Main theorem}
The main theorem states, in slightly other words, that for any $\cC^1$ neighbourhood of a function $f\in\cC^1_{=m}(U,\R^n)$ there exist perturbation $f_\eps$ such that an $f_\eps(\Sigma)$ has zero measure. 
\main
\begin{proof}
We will prove that for every $\eps>0$ and $f\in\cA(\Sigma)$ there exist $f_\eps\in \cC^1_{=m}$ such that $f_\eps\notin\cA(\Sigma)$ and $\|f-f_\eps\|_{\cC^1(U)}\leq\eps$, which will imply that the interior of 
$\cA(\Sigma)$ is empty.
Function $f_\eps$ will be the limit of a sequence $f_n$ contained in an $\eps$-neighbourhood of $f$ and such that $\cH^n(f_n(\Sigma))\overset{n\rightarrow\infty}{\longrightarrow} 0$.

From \cref{localver} we already know that locally \cref{thm:main} is true.  The problem is that diffeomorphisms $\Xi_\theta$ do not agree on overlapping sets $U(x,r)$ and $U(\tilde{x},\tilde{r})$.  Thus we are forced to work on disjoint sets. Additionally diffeomorphisms $\Xi_\theta$ acts non-trivially on the boundary of $U(x,r)$, hence those sets have to be not only disjoint but also have positive distance between each other. 

The following lemma enables us to prolong a mapping $\Xi_\theta$ from the \cref{localver} to a larger set. 
For a subset $\cO$ contained in $U(x,r)$, provided that is is far enough from the boundary of $U(x,r)$, we will find a diffeomorphism 
which acts like $\Xi_\theta$ on $\cO$ and equals to identity outside some neighbourhood of the set $\cO$.
\begin{lemma}\label{keylemma}
If a compact set $U$ is $\cC^1-$diffeomorphic to a closed ball $\varphi\colon\overline{U}\rightarrow \overline{B}(0,r)$ and 
\begin{enumerate}[label=(\alph*)]
 \item $\cO\subset U\setminus B_\mu(\partial U),$ for some $\mu>0$;
 \item $\eta>0$,
\end{enumerate}
then 
there exists $\rho>0$ such that for any $\theta\in B_{SO}(\id,\rho)\setminus\lbrace\id\rbrace$ there exists a
diffeomorphism $$\zeta\coloneqq\zeta(\mu,\eta,\theta)\colon U\rightarrow U$$ with the following properties:
\begin{enumerate}[label=(\roman*)]
 \item $\zeta(B_{\mu/4}(\cO))\subseteq B_{\mu/2}(\cO)$ and $\zeta(y)=\Xi_\theta(y)$ for $y\in B_{\mu/4}(\cO)$;
 \item $\zeta(y)=y$ for $y\in U\setminus B_{3\mu/4}(\cO)$;
 \item $\|\zeta -\id\|_{\cC^1(U)} \leq \eta $. \label{lab3}
\end{enumerate}
\end{lemma}
\begin{proof}
For an element $\theta\in SO(n)$ sufficiently close to $\id$ (before cut locus) there exists the unique path 
$$\exp(t\:X_\theta)=\theta_t:\R\rightarrow SO(n)\quad \text{for}\quad X_\theta\in\mathfrak{so}(n)$$
connecting $\id$ and $\theta$ that is $\theta_1=\theta$, $\theta_0=\id$.
Mapping $t\mapsto \exp(t\:X_\theta)$ is a group endomorphism $\R\rightarrow\text{Diff}(B(0,r))$, in other words it is \emph{a flow}, and induces a vector field 
$$X(b)=\frac{d}{dt}\exp(t\:X_\theta)(b)\vert_{t=0}\quad\text{for } b\in B(0,r).$$
The aim now is to produce a vector field $V_\theta\in\Gamma(TU)$ such that its vector flow gives the mapping $\Xi_\theta(y)=\varphi_{V_\theta}(1,y)$.
Notice that the field $X$ can be transported back to the set $U$ via diffeomorphism $\phi$ or again using the fact that the
flow $\Xi_{\theta_t}:\R\times U\rightarrow U$ generates a vector field 
\begin{align}
 V_\theta(y)\coloneqq\frac{d}{dt}\Xi_{\theta_t}(y)\vert_{t=0} &= \frac{d}{dt}\phi^{-1}\circ\theta_t\circ\phi(y)\vert_{t=0} \nonumber\\
 &=D\phi^{-1}(\theta_0\circ\phi(y))\frac{d}{dt}\theta_t(\phi(y))\vert_{t=0}\nonumber\\
 &=D\phi^{-1}(\phi(y)).X(\phi(y))\nonumber\\
 &={\phi^{-1}}_\ast(\phi^\ast X)(y)\qquad\text{for }y\in U.
\end{align}
Multiplying the field $V_{\theta}$ by a smooth cutoff function of an appropriate neighbourhood of the set $\cO$ we obtain a flow that at small enough time will 
do exactly what was formulated in the points $(i)$ to $(iii)$.
Let  
 \begin{equation}
  W_\theta(y)\coloneqq (\vartheta_{\mu/8}\ast\mathds{1}_{B_{5\mu/8}(\cO)})(y)\cdot V_{\theta}(y).
 \end{equation}
 By $\varphi_W{_\theta}(t,y)$ we denote the trajectory of the point $y$ at the time $t$ under the flow of the field $W_\theta$. 
 Since for every point $y$ in $\mu/2$-neighbourhood of the set $\cO$, denoted $B_{\mu/2}(\cO)$, the support of the mollifier $\vartheta_{\mu/8}(\cdot-y)$ is contained 
 in $(\mu/2+\mu/8)$-neighbourhood of $\cO$, the field $W_\theta$ is equal to $V_\theta$ on $B_{\mu/2}(\cO)$.
 Therefore, there exists $t$ such that the flow $\varphi_{W_\theta}(s,y)=\Xi_{\theta_s}(y)$ for $y\in B_{\mu/4}(\cO)$ and $s\leq t$. Hence gives the proof of the point $(i)$.
 
  The field $W_\theta$ is equal to zero outside $\spt (\vartheta_{\mu/8}\ast\mathds{1}_{B_{5\mu/8}(\cO)}) \subseteq B_{3\mu/4}(\cO)$, hence for 
 $y\in B_{3\mu/4}(\cO)^c$ we have $\varphi_W(t,y)=y$ and the point $(ii)$.
 
 Field $W_\theta$ is Lipschitz on compact set $\overline{U}$ and equal to identity on the neighbourhood of the boundary, and generates one-parameter group of diffeomorphisms 
 that exists for times $t\in(-\eps,\eps)$ for some positive $\eps$. Differentiable dependence on initial conditions guaranties that we will find time $t$ so small that \ref{lab3} is satisfied for all $s<t$.
 
 So far for some fixed $\theta$ we found $t$ such that $\theta_t$ fulfils $(i),(ii)$ and $(iii)$. Observe that the mapping 
\begin{align}
 SO(n)&\longrightarrow \text{Diff}_{\cC^1}(U) \nonumber\\
 \theta&\mapsto \varphi_{W_\theta}(1,\cdot) \nonumber
\end{align}
is continuous, hence $\theta\mapsto \|\varphi_{W_\theta}(1,\cdot)-\id\|_{\cC^1}$ is also continuous. 
Therefore, there exist and open neighbourhood of $\id$ such that properties $(i),(ii)$ and $(iii)$ are satisfied.
%
 \end{proof}
\emph{Outline of the proof of \cref{thm:main}}.  First we shall construct a countable family $\mathfrak{U}$ of pairwise disjoint connected open sets $\left\lbrace U_{i}\right\rbrace$, such that every $U_i$ is contained in some domain 
$U(x,r)=U(x(i),r(i))$. The set $U$ will be covered by the sum of closures of $U_i$  such that the measure $\cH^m(\Sigma\cap\partial U_i)=0$ for every $i$.

Consider an element $U_1$ of the family $\mathfrak{U}$. By \cref{keylemma} there exists a diffeomorphism $\zeta_{1,1}$ such that on the set 
$\cO\coloneqq U_1\setminus B_\mu(\partial U_1)$ it is equal to $\Xi_\theta$. The map $\Xi_\theta$ is chosen, using lemma \cref{localver}, in such a way that the $m$-measure of the image of $\cO\cap\Sigma$ under $f\circ\zeta_{1,1}$ is zero, i.e. $\cH^{m}(f\circ\Xi_{\theta}(\cO\cap\Sigma))=0$.
In the set $\zeta_{1,1}(U_1)$ the only possible place for positive $\cH^m$ measure of $f\circ\zeta_{1,1}(\Sigma)$ is outside $\zeta_{1,1}(\cO)$. Define this set by $U_{1,2}\coloneqq U_1\setminus\zeta_{1,1}(\cO)$.  
The idea is to repeat the above procedure:  apply \cref{keylemma} to the set $U_{1,2}\setminus B_{\mu_2}(\partial U_{1,2})$.
This yields a diffeomorphism $\zeta_{1,2}$ which slightly moves the set that is in $U_{1,2}$ and further than $\mu_2$ away from the boundary of $U_{1,2}$.

The sequence $f_n$ appears as the composition 
\begin{align*}
 f_1&=f\circ\zeta_{1,1}\\
 f_2&=f\circ\zeta_{1,2}\circ\zeta_{1,1} \\
    &\vdots\\
 f_n&=f\circ\zeta_{1,n}\circ\ldots\circ\zeta_{1,1}.
\end{align*}
It will turn out that this sequence is a finite composition for almost all points $u\in U_1$.
We have to check that this sequence is converging to some map $f_\eps$. 
Here we end the outline and start the proof with a construction of the cover $\mathfrak{U}$. 

 1.$\quad$ Construction of the family $\mathfrak{U}$.\newline From now on let the family
 \begin{equation}\label{prefamily}
  \widetilde{\mathfrak{B}}\subset\lbrace U(x,r)\rbrace_{x\in U,0<r<r_x}
 \end{equation}
 stand \emph{only} for those sets introduced in \cref{localver} that fulfil property $\cH^m(\Sigma\cap\partial U(x,r))=0$. 
 Note that for every $x$ in $U$ at most countable set of radii $r\in(0,r_x)$ does not have the above property. From this family choose a countable collection $\mathfrak{B}\subset\widetilde{\mathfrak{B}}$,
 covering $U$ and denote them $\mathfrak{B}=(V_i)_{i\in\bN}$. Consider the following collection of open sets 
 \begin{equation}
\widetilde{\mathfrak{U}}=\left\lbrace \into{V_1},\:\into{V_2}\setminus \overline{V_1},\:\into{V_3}\setminus\overline{(V_1\cup V_2)},\:\ldots,\into{V_n}\setminus\overline{\bigcup\limits_{i<n}V_i},\:\ldots\right\rbrace.
 \end{equation}
 Finally define $\mathfrak{U}$ as the family of connected components of $\widetilde{\mathfrak{U}}$. Denote elements of $\mathfrak{U}$ by $(U_i)_{i\in\bN}$.
 \begin{example}
  On the plane, let $V_1=B(0,1)$, $V_2=(-2,2)\times(-1,1)$ and $\mathfrak{B}=\lbrace V_1, V_2\rbrace$. In this case $\widetilde{\mathfrak{U}}$ consists of two sets: the ball $B(0,1)$ and the rectangle $(-1,1)\times(-1,1)$ 
  with ball cut out. Then family $\mathfrak{U}$ has three elements: the ball, left remains from rectangle after cutting out the ball and symmetric right part. 
 \end{example}
 Already $\into{V_2}\setminus \overline{V_1}$ can consist of countably many connected components. The support of the measure $\cH^m\mres\Sigma$ is contained in sum of sets from $\mathfrak{U}$ because of the condition
 imposed on \eqref{prefamily}. Since  
 \begin{equation}
  \forall_{j\in\bN}\:\exists_{n\in\bN}\quad U_j\overset{conn}{\subset}\into{V_n}\setminus\bigcup\limits_{i<n}V_i,
 \end{equation}
the boundary of $U_i$ can be divided into finite sum
\begin{equation}
 \forall_{j\in\bN}\:\exists_{n\in\bN} \quad \partial U_j\subset \partial V_n\cup\partial V_{n-1}\cup\ldots\cup \partial V_1.
\end{equation}
Note that the $\partial U_j\cap \partial V_i$ can have infinitely many connected components.

 2.$\quad$ $\cH^m\mres\Sigma$-measure of collars around the boundaries of elements in $\mathfrak{U}$.\newline
 For a continuously differentiable, compact $(n-1)$ dimensional manifold $\partial V_k$ the limit
 \begin{equation}
  \cH^m(\Sigma \cap B_\mu(\partial V_k))\longrightarrow 0 \quad \text{as}\quad\mu\rightarrow 0.
 \end{equation}
Conversely, assume $\lim\limits_{\mu\rightarrow 0}\cH^m(\Sigma \cap B_\mu(\partial V_k))> 0$. 
Then since $\Sigma$ has finite measure we have 
$$\cH^m(\Sigma\cap \partial V_k)=\cH^m(\Sigma\cap \bigcap\limits_{\mu>0} B_\mu(\partial V_k))
=
\lim\limits_{\mu\rightarrow 0}\cH^m(\Sigma \cap B_\mu(\partial V_k))> 0$$ 
which contradicts assumption on \eqref{prefamily}.

 Fix $U_i\in\mathfrak{U}$ and let $$\sigma_i=\cH^m(\Sigma\cap U_i).$$ Since $\partial U_i \subset \bigcup_{k<n} \partial V_k $ for some $n=n(i)$ then 
\begin{equation}\label{collarmeasure}
 \cH^m\mres\Sigma(B_\mu(\partial U_i))\leq\sum\limits_{k=1}^n\cH^m\mres\Sigma (B_\mu(V_k))\longrightarrow 0\quad \text{as}\quad \mu\rightarrow 0.
\end{equation}
Therefore there exist $\:\mu_{i,1}>0\:$ such that
\begin{equation}\label{mu:bdyass}
\quad\cH^m\mres(U_i\cap\Sigma)\left(\vphantom{\int} B_{\mu_{i,1}}(\partial U_i)\right)<\sigma_i/3
\quad\text{and}\quad\cH^m\mres(U_i\cap\Sigma)\left(\vphantom{\int}\partial(B_{\mu_{i,1}}(\partial U_i))\right)=0.
\end{equation} 
We have second assertion because $\left\lbrace\partial B_{s}(\partial U_{i})\right\rbrace_{s\in[0,1]}$ is an uncountable, pairwise disjoint family of sets which are $(n-1)$-rectifiable for almost all $s$. 

3.$\quad$ Using the \cref{keylemma}.\newline
We need to control the $\cH^m\mres\Sigma$-measure in the collars around the boundary, there we do not have property $\cH^m(f\circ\zeta(\Sigma))=0$. 
Since $f$ is $\cC^1$ up to the boundary of $V_i$, without loss of generality, we can assume that $\text{Lip}(f)\leq 1$ and bound the measure of $\Sigma$ in the domain
instead of the measure of the image. 
Let 
\begin{equation}\label{epsass}
 \eps=\sum\limits_{k=1}^\infty\eps_k \quad \text{with }\eps_k>0
\end{equation}
Apply \cref{keylemma} to the set $\cO_i=U_i\setminus B_{\mu_{i,1}}(\partial U_i)$ for fixed $U_i\in\mathfrak{U}$, with a positive 
\begin{equation}\label{conditioneta}
 \eta<\eps_1/3
\end{equation}
and $t$ in vector flow $\varphi_{W_\theta}(t,\cdot)=\zeta$  so small that 
\begin{equation}
 \cH^m\mres\zeta(\Sigma)\left(\vphantom{\int}\zeta( U_i\cap B_{\mu_{i,1}}(\partial U_i))\right)<\sigma_i/2.
\end{equation}
Note that last step may increase the measure of $\zeta(\Sigma)$ in the deformed collar around boundary of the set $U_i$ compared to \eqref{mu:bdyass}. 
Due to the continuity of function $t\mapsto \cH^m(\varphi_{W_\theta}(t,\Sigma))$, which follows from \eqref{mesoflip}, it is possible to bound 
this increase by $\sigma_i/2$.
In the \cref{keylemma} we deform a set $U$ diffeomorphic to some ball, here we know that $U_i$ is contained in some set $V_j$ which is diffeomorphic to a ball, 
see \eqref{prefamily}.
Due to \cref{localver} we can choose a map $\theta$ such that $\zeta$ is equal to $\Xi_\theta$ on the set $\cO_i$ and such that    
\begin{equation}
 \cH^m\left(\vphantom{\int}f\circ\zeta(\Sigma\cap \cO_i )\right)=0.
\end{equation}
As a result the measure of $\Sigma$ vanished in the set that is away from boundary of $U_i$.

4.$\quad$ Iteration\newline 
Now we will iterate above construction. 
 We are working with one element of the family $\mathfrak{U}$. 
 Denote $U_{1}\coloneqq U_i$, $\cO_{1}\coloneqq\cO_i$, $\mu_1\coloneqq\mu_{i,1}$ and also denote by $\zeta_{1}$ mapping $\zeta$ obtained in the previous point. 
Let  
\begin{align}
 U_{2}&=U_{1}\cap\zeta_{1}\left(
 B_{\mu_{1}}(\partial U_{1})\right) \\
\text{or shortly}\quad &=U_{1}\setminus\zeta_{1}(\cO_{1}).
\end{align}
This is the only possible place where, after composition with $f$, set $\zeta_{1}(\Sigma)$ yields positive measure.
In the previous step the boundary of $U_1$ was $(n-1)-$rectifiable and we clearly had \eqref{collarmeasure}. 
Notice that the set changed, we are working now with $\zeta_{1}(\Sigma)$ inside a little collar around boundary of $U_{1}$ distorted by $\zeta_1$.
Again we use fact that boundary of $B_\mu(A)$ is rectifiable for any $A$ and almost all $\mu$.
Hence we can choose $\mu_{2}$ such that 
\begin{align}
 \cH^m\mres(U_{2}\cap\zeta_{1}(\Sigma))&\left(\vphantom{\int}\partial(B_{\mu_{2}}(\partial U_{2}))\right)=0\quad\text{and} \\
 \cH^m\mres(U_{2}\cap\zeta_{1}(\Sigma))&\left(\vphantom{\int} B_{\mu_{2}}(\partial U_{2})\right)<\sigma/3^2
\end{align}
 Apply lemma \cref{keylemma} to the set $\cO_{2}= U_{2}\setminus B_{\mu_{2}}\left(\partial U_{2}\right)$
with $\eta<\eps_2/3$ and $t$ in the flow such that 
 \begin{equation}
  \cH^m\left(\vphantom{\int}\zeta_{2}(\zeta_{1}(\Sigma)\cap U_{2}\cap B_{\mu_{2}}(\partial U_{2}))\right)=\cH^m\left(\vphantom{\int}\zeta_{2}(\zeta_{1}(\Sigma)\cap U_{2}\setminus\cO_{2})\right)<\sigma/2^2.
 \end{equation}
 We obtained the map $f_2=f\circ\zeta_{2}\circ\zeta_{1}$ with $\cH^m(f_2(\Sigma\cap U_1))\leq \sigma/2^2$ as we assumed $\text{Lip}(f)\leq 1$. 
 Moreover
\begin{align}
 \|\zeta_{2}\circ\zeta_{1}-\id\|_{\cC^1(U_1)}&=\|\zeta_{2}\circ\zeta_{1}-\zeta_{1}+\zeta_{1}-\id\| \nonumber\\
 &\leq\|\zeta_{2}\circ\zeta_{1}-\zeta_{1}\|+\|\zeta_{1}-\id\| \nonumber\\
 &\leq \eps_2 +\eps_1.
\end{align}
If $f,g:U\rightarrow U$ are two diffeomorphisms $\eps_f$ and $\eps_g$ close to identity in $\cC^1$ respectively, then
\begin{align}\label{diffineq}
 \|f\circ g - g \|_{\cC^1}&=\|f\circ g -g\|_{\cC^0}+\|D(f\circ g)-D(g)\|_{\cC^0} \nonumber \\
 &= \|f-\id\|_{\cC^0} +\|Df(g).Dg-D(g)\|_{\cC^0} \nonumber \\
 &\leq \eps_f + \|Df(g)-\id\|_{\cC^0} \|Dg\|_{\cC^0}\quad \text{for}\quad \|Dg\|\in 1\pm\eps_g \nonumber \\
 &\leq \eps_f + \eps_f(1+\eps_g)
\end{align}
and if one take $\eps_f <\eps/3$ and $\eps_g<1/3$ then $\|f\circ g - g \|_{\cC^1}<\eps$. This is the reason for \eqref{conditioneta}.

5.$\quad$ The $n^{\text{th}}-$step.
First we define  $U_{n}=U_{n-1}\setminus\zeta_{n-1}(\cO_{n-1})$. Take a collar $$\text{col}_{\mu_{n}}(U_{n})=U_{n}\cap B_{\mu_{n}}(\partial U_{n})$$ around boundary of $U_{n}$ 
of width $\mu_{n}$. Since boundary of $U_{n}$ and the boundary of the collar 
are rectifiable sets for almost all $\mu_{n}$, we can take $\mu_{n}$ such that
\begin{align}
\cH^m\mres\left(
\zeta_{n-1}\circ\ldots\circ\zeta_{1}(\Sigma)\right)&\left(\partial\:\text{col}_{\mu_{n}}(U_{n})\right)=0\quad \text{and}\\
\cH^m\mres\left(
\zeta_{n-1}\circ\ldots\circ\zeta_{1}(\Sigma)\right)&\left(\text{col}_{\mu_{n}}(U_{n})\right)<\sigma/3^n.
\end{align}
Use the \cref{keylemma} to the set $\cO_{n}=U_{n}\setminus \text{col}_{\mu_{n}}(U_{n})$ with $\eta<\eps_n/3$ and $t$ in the flow so small that 
\begin{equation}
 \cH^m\left(\vphantom{\int}\zeta_{n}\left(\vphantom{{1^1}^1}\zeta_{n-1}\circ\ldots\circ\zeta_{1}(\Sigma)\cap\text{col}_{n}(U_{n})\right)\right)\leq\sigma/2^n
\end{equation}
We estimate the distance of the above composition to identity map.
\begin{align}\label{disttoid}
 \|\zeta_{n}\circ\ldots\circ\zeta_{1}-\id\|_{\cC^1}&=\|\zeta_{n}\circ\ldots\circ\zeta_{1}+(-\zeta_{n-1}\circ\ldots\circ\zeta_{1}+\zeta_{n-1}\circ\ldots\circ\zeta_{1})+ \nonumber\\
 &\quad+(-\zeta_{n-2}\circ\ldots\circ\zeta_{1}+\zeta_{n-2}\circ\ldots\circ\zeta_{1})+\ldots+(-\zeta_{1}+\zeta_{1})-\id \|\nonumber\\
 &\leq \|\zeta_{n}\circ\ldots\circ\zeta_{1}-\zeta_{n-1}\circ\ldots\circ\zeta_{1}\| \nonumber\\
 &\quad+\|\zeta_{n-1}\circ\ldots\circ\zeta_{1}-\zeta_{n-2}\circ\ldots\circ\zeta_{1}\| \nonumber\\
 &\qquad\vdots \\
 &\quad+\|\zeta_{2}\circ\zeta_{1}-\zeta_{1}\| \nonumber\\
 &\quad+\|\zeta_{1}-\id\| \nonumber\\
 \text{from }\eqref{diffineq}\qquad&\leq \sum\limits_{i=1}^n\eps_i<\eps. \nonumber
 \end{align}

 6.$\quad$ Mapping in the limit.\newline 
 Let 
 \begin{equation}
 \zeta\colon U\longrightarrow \R^n \quad\text{be defined by}\quad \zeta(u)=\lim_{n\rightarrow\infty}\zeta_{n}\circ\ldots\circ\zeta_1(u)
 \end{equation}
 Space $\cC^1(U,\R^n)$ is a Banach space, if the sequence $\zeta_n$ is a Cauchy sequence then it will imply that $\zeta\in\cC^1(U,\R^n)$.
 Let $\zeta_n^\circ\coloneqq \zeta_n\circ\ldots\circ\zeta_1$. Almost the same as in \eqref{disttoid}, for any $\varrho>0$
 \begin{align}
  \|\zeta_n^\circ-\zeta_m^\circ\|_{\cC^1}&=\|\zeta_n^\circ-\zeta_{n-1}^\circ+\zeta_{n-1}^\circ-\ldots-\zeta_{m+1}^\circ+\zeta_{m+1}^\circ-\zeta_m^\circ\| \\
  &\leq \|\zeta_n^\circ-\zeta_{n-1}^\circ\|+\|\zeta_{n-1}^\circ-\zeta_{n-2}^\circ\|+\ldots+\|\zeta_{m+1}^\circ-\zeta_m^\circ\| \\
  \text{from }\eqref{diffineq}\quad &\leq\eps_n+\eps_{n-1}+\ldots+\eps_{m+1}+\eps_m 
 \end{align}
 Since the series $\sum_k\eps_k$ is convergent, one can find an integer $N$ such that for $n,m>N$ the condition $$ \|\zeta_n^\circ-\zeta_m^\circ\|_{\cC^1}\leq\varrho$$ is satisfied, proving that 
 $\zeta\in\cC^1(U,\R^n)$. Notice that $D\zeta =\id$ on the boundary because all $D\zeta_n$ are equal to the identity on the boundary.

 7.$\quad$Construction for the whole set $U$.\newline
 Having a construction on one element of family $\mathfrak{U}$ we easily produce the sequence for the whole domain $U$, because on every element of the cover $\mathfrak{U}$ the modification can be done separately. From now on $U_i$ denote elements of the family $\mathfrak{U}$ again
 and $\zeta_{i,n}$ is the mapping at the $n^{\text{th}}-$step in the element $U_i$. The mapping $\zeta_{i,n}$ can be extended with $\id$ to $U\setminus U_i$. 
 Define the sequence $(\xi_n)_{n=1,2,\ldots}$ by
 \begin{align}
 \xi_1&=\zeta_{1,1} \nonumber\\
 \xi_2&=\zeta_{1,2}\circ\zeta_{1,1}\circ\zeta_{2,1} \nonumber\\
 \xi_3&=\zeta_{1,3}\circ\zeta_{1,2}\circ\zeta_{1,1}\circ\zeta_{2,1}\circ\zeta_{1,1}\circ\zeta_{3,1} \nonumber\\
 \xi_4&=\zeta_{1,4}^\circ\circ\zeta_{2,3}^\circ\circ\zeta_{3,2}^\circ\circ\zeta_{4,1}^\circ  \nonumber\\
 &\qquad\vdots \nonumber\\
\xi_n&=\zeta_{1,n}^\circ\circ\zeta_{2,n-1}^\circ\circ\ldots\circ\zeta_{n-1,2}^\circ\circ\zeta_{n,1}\nonumber
\end{align}
where $\zeta_{i,k}^\circ =\zeta_{i,k}\circ\zeta_{i,k-1}\circ\ldots\circ\zeta_{i,1}$.
The composition sign between functions in the last line is, in fact, a long list of identities because $$\text{spt}(\zeta_{k,l}-\id)\cap\text{spt}(\zeta_{m,n}-\id)= \varnothing \quad\text{for}\quad k\neq m .$$ 
On every component $U_i$ sequence $\xi_n$ is converging to corresponding $\zeta$. After composition with the map $f$ the measure of $\Sigma\cap U_i$ is zero for every $i$.
\end{proof}
\subsection{Potential extension of the theorem}\label{sec:lowrank}
Mappings with $m$ \emph{or lower} rank. 
\noindent Under the conjecture 
\begin{equation}\label{trueornot}
 \cC^\infty_{\leq m}(U,\R^n)\overset{dense}{\subset}\cC^1_{\leq m}(U,\R^n)
\end{equation}
same is true for broader class of mappings $\cC^1_{\leq m}(U,\R^n)$. The argument is as follows.
Let $\|f_\eps -f\|_{\cC^1}\leq \eps$ and $f_\eps\in \cC^\infty_{\leq m}$. Split the domain into two sets 
$U_m=\lbrace x \:|\:r(Df_\eps(x))=m\rbrace$ and $U\setminus U_m$.
On the set $U_m$ do the same procedure as for maps belonging to $\cC^1_{=m}$, on the remaining set use the Sard theorem 
\cite[Theorem 3.4.3]{federer}. It states that 
if $B=\lbrace x\:\vert \: r(Df_\eps(x))\leq m-1\rbrace $, the set $U\subset\R^n$ and the mapping $f$ is of class $k$, then
$$\cH^{m-1+(n-(m-1))/k}(f(B))=0.$$
Apply it to $f_\eps$ and notice that if $k>n-m+1$ then $m-1+(n-(m-1))/k<m$ and then the $\cH^m$ measure of the image 
of the whole set $B$ equals zero, hence $\cH^m(f_\eps(B\cap\Sigma))=0$ as $f_\eps$ belongs to any class $\cC^k$. 
Therefore if \eqref{trueornot} is true then the thing we have to do is to modify map on the open set where rank of the 
Jacobian matrix is equal to $m$. 
\section{Acknowledgements}
The author wants to thank Paweł Strzelecki and Sławomir Kolasiński for having drawn his attention to this problem. He would like to thank the audience 
of the ``illegal seminar'' for constant support and useful discussions. He would also like thank two, one local and one non-local, early stage readers for their 
invaluable help in improving the manuscript. The author was partially supported by NCN Grant no. 2013/10/M/ST1/00416 \emph{Geometric curvature energies 
for subsets of the Euclidean space}.

\bibliographystyle{alpha}
\bibliography{bib}

\begin{thebibliography}{{Pug}16}

\bibitem[Alm68]{almExi}
Jr. F.~J. Almgren.
\newblock {Existence and regularity almost everywhere of solutions to elliptic
  variational problems among surfaces of varying topological type and
  singularity structure}.
\newblock {\em Ann. of Math. (2)}, 87:321--391, 1968.

\bibitem[Bes39]{Besico3}
A.~S. Besicovitch.
\newblock {On the fundamental geometrical properties of linearly measurable
  plane sets of points ({III})}.
\newblock {\em Math. Ann.}, 116(1):349--357, 1939.

\bibitem[DS00]{unif-rect}
Guy David and Stephen Semmes.
\newblock {Uniform rectifiability and quasiminimizing sets of arbitrary
  codimension}.
\newblock {\em Mem. Amer. Math. Soc.}, 144(687):viii+132, 2000.

\bibitem[Fed47]{Fed47}
Herbert Federer.
\newblock {The {$(\varphi,k)$} rectifiable subsets of {$n$}-space}.
\newblock {\em Trans. Amer. Soc.}, 62:114--192, 1947.

\bibitem[Fed69]{federer}
Herbert Federer.
\newblock {\em {Geometric measure theory}}.
\newblock {Die Grundlehren der mathematischen Wissenschaften, Band 153}.
  Springer-Verlag New York Inc., New York, 1969.

\bibitem[FF60]{ff60}
Herbert Federer and Wendell~H. Fleming.
\newblock {Normal and integral currents}.
\newblock {\em Ann. of Math. (2)}, 72:458--520, 1960.

\bibitem[Mat95]{Mattila1}
Pertti Mattila.
\newblock {\em {Geometry of sets and measures in {E}uclidean spaces}},
  volume~44 of {\em {Cambridge Studies in Advanced Mathematics}}.
\newblock Cambridge University Press, Cambridge, 1995.
\newblock Fractals and rectifiability.

\bibitem[{Pug}16]{Pughlocal}
H.~{Pugh}.
\newblock {A Localized Besicovitch-Federer Projection Theorem}.
\newblock {\em ArXiv e-prints}, July 2016.

\bibitem[Spi65]{Spivak}
Michael Spivak.
\newblock {\em {Calculus on manifolds. {A} modern approach to classical
  theorems of advanced calculus}}.
\newblock W. A. Benjamin, Inc., New York-Amsterdam, 1965.

\end{thebibliography}
\end{document}